\documentclass[a4paper,10pt]{article}
\usepackage[utf8x]{inputenc}
\usepackage[colorlinks=true, pdfstartview=FitV, linkcolor=purple, citecolor=purple]{hyperref}
\usepackage{amsmath}
\usepackage{amsthm}
\usepackage{mathdots}
\usepackage{cleveref}
\usepackage{bbm}
\usepackage{graphicx, tikz, pgfplots}
\usetikzlibrary{shapes, positioning, arrows.meta,calc}
\usepackage[margin=0.8in]{geometry}
\usepackage{amssymb}
\usepackage{tabularx}
\usepackage{float}
\usepackage{import}
\usepackage{enumitem}
\usepackage[toc,page]{appendix}
\usepackage{listings}
\usepackage{xcolor} 
\usepackage{makeidx}
\makeindex

\newtheorem{theorem}{Theorem}[section]
\newtheorem{lemma}[theorem]{Lemma}
\newtheorem{proposition}[theorem]{Proposition}
\newtheorem{corollary}[theorem]{Corollary}

\newtheorem{remark}{Remark}

\newcommand\floor[1]{\left\lfloor #1 \right\rfloor}

\title{On Convergents of Proper Continued Fractions}

\author{Niels Langeveld and David Ralston}
\begin{document}
\maketitle
\begin{abstract}
    Proper continued fractions are generalized continued fractions with positive integer numerators $a_i$ and integer denominators with $b_i\geq a_i$. In this paper we study the strength of approximation of irrational numbers to their convergents and classify which pairs of integers $p,q$ yield a convergent $p/q$ to some irrational $x$. Notably, we reduce the problem to finding convergence only of index one and two. We completely classify the possible choices for convergents of odd index and provide a near-complete classification for even index. We furthermore propose a natural two-dimensional generalization of the classical Gauss map as a method for dynamically generating all possible expansions and establish ergodicity of this map.
\end{abstract}

\section{Introduction and History}\label{sec: introduction}
Proper continued fractions (PCF) where introduced by Leighton in 1940 as a generalisation of regular continued fraction expansions \cite{L40}. We will presume some familiarity with \textit{regular} continued fractions throughout. Leighton showed that for any real number $x\in(0,1]$ and any sequence of positive integers $(a_n)_{n\geq 1}$ there exists a sequence of integers $(b_n)_{n\geq 1}$ with $b_n \geq a_n$ for all $n\in\mathbb{N}$ such that 
\begin{equation}\label{eq:propercf}
    x= \cfrac{a_1}{b_1+\cfrac{a_2}{b_2+\ddots \cfrac{a_n}{b_n+\ddots}}}
\end{equation}
where this continued fraction is finite if and only if $x$ is rational. Just as for regular continued fractions we can write for irrational $x$
\[x = \lim_{n \rightarrow \infty} \cfrac{a_1}{b_1+\cfrac{a_2}{b_2+\ddots+\frac{a_n}{b_n}}}=\lim_{n \rightarrow \infty} \frac{p_n}{q_n},\]
with $c_n=\frac{p_n}{q_n}$. The introduction of variable numerators in PCF changes the typical recurrence relations for $p_n$ and $q_n$:
\begin{equation}\label{eq:rec for pn and qn}
    \begin{aligned}
        p_{-1}=1,&\, p_0=0, & p_n=b_np_{n-1}+a_np_{n-2},~\text{for } n\geq 1,\\
    q_{-1}=0, &\, q_0=1, & q_n=b_nq_{n-1}+a_nq_{n-2},~\text{for } n\geq 1.
    \end{aligned}
\end{equation}
Leighton proved in \cite{L40} that $\lim_{n\rightarrow \infty} c_n=x$ and we call $c_n$ the $n^{\text{th}}$ convergent of $x$ using the given sequence of numerators $a_i$.

For regular continued fractions $p_n$ and $q_n$ are always co-prime. For these proper continued fractions this is typically not the case, and it is \emph{not} our convention to reduce such fractions when they appear. Statements regarding $p_n/q_n$ are always taken to refer directly to the numbers generated by the recurrence relations \eqref{eq:rec for pn and qn}. When this ratio as a reduced rational number is needed, we will refer to it as $c_n$. In \Cref{sec: basic properties} we will investigate the most basic properties of PCF expansions of an arbitrary $x$, such as expansions of rational numbers and the rate of convergence of $p_n/q_n$ to $x$.

A celebrated and foundational result in the theory of continued fractions is that the convergents generated by the regular continued fraction expansion of an irrational $x$ (with the possible exception of the first one) are \textit{very} close to $x$ (specifically, $p_n/q_n$ is within $1/q_n^2$ of $x$), and that the convergents are exactly the \textit{best approximations of the second kind} to $x$. As usually paraphrased, the regular continued fraction expansion of $x$ exactly gives those rational numbers which are the ``best rational approximations" to $x$. Since the regular continued fraction expansion of $x$ is just one possible proper continued fraction expansion, it is clear that the rational numbers which may appear as PCF convergents to $x$ form a weaker approximation to $x$. In \Cref{sec: approximation results} we investigate the same question: how well do PCF convergents approximate $x$, and what characterizes the set of all possible convergents? 

For an irrational $x$, given a sequence of positive integers $a_i$, the unique PCF expansion of $x$ with \textit{numerators} $a_i$ is found recursively by setting $x_0=x$, and then
\[b_i = \floor{\frac{a_i}{x_{i-1}}}, \qquad x_i = \frac{a_i}{x_{i-1}}-b_i.\]
With this observation in mind, we recall the map $T_N:(0,1] \mapsto [0,1)$ as defined in \cite{DKW13}:
\begin{equation}\label{eqn:T_N}
T_N(x) = \frac{N}{x} - \left\lfloor\frac{N}{x}\right\rfloor\end{equation}
For $N=1$ this is simply the Gauss map, which we denote as simply $T$. When all numerators are equal to $N$, the map $T_N$ is a natural analogue to the Gauss map. For proper continued fractions in general, where any irrational $x \in (0,1)$ has uncountably many different possible expansions, a different map will be needed. In \Cref{sec: new map and ergodicity} we propose a natural map and establish ergodicity with respect to a measure mutually absolutely continuous with respect to Lebesgue measure.

\section{Basic Properties}\label{sec: basic properties}
First, we will show some results regarding PCF convergents of rational numbers. Let $x_0=\frac{t_0}{s_0}$ be a rational number in $(0,1)$ such that $t_0$ and $s_0$ are co-prime.

\[
T_N\left(\frac{t_0}{s_0}\right)=\frac{N}{t_0/s_0}-\floor{\frac{N s_0}{t_0}} =\left\{\frac{N s_0}{t_0} \right\}
\]
where $\{ \cdot\}$ means the fractional part. As $s_0, t_0$ are relatively prime, we may choose $N$ to the numerator of $T_N(x_0)$ any value between $0$ and $t_0-1$. Specifically:
\begin{equation}\label{eq:TNimages}
    \bigcup_{N\in\mathbb{N}} T_N\left(\frac{t_0}{s_0}\right)=\bigcup_{N=1}^{t_0} T_N\left(\frac{t_0}{s_0}\right) =\left\{0,\frac{1}{t_0},\frac{2}{t_0}, \ldots , \frac{t_0-1}{t_0} \right\}
\end{equation}

From \eqref{eq:TNimages} we can conclude the following:
\begin{enumerate}
    \item All expansions of $x_0$ are finite as the denominator of $T_N(x_0)$ is strictly smaller that that of $x_0$.
    \item The longest length a PCF expansion of $x_0=\frac{t_0}{s_0}$ can have is $t_0$.
    \item By selecting numerators appropriately, for any length between one and $t_0$, there is some PCF expansion of $x_0$ of that length.
\end{enumerate}
We must be careful in describing the length of a PCF expansion of some rational number. An example of a longest length continued fraction is
\[
\frac{5}{6}=\frac{4}{\displaystyle 4+\frac{3}{\displaystyle  3+\frac{2}{\displaystyle  2+\frac{1}{\displaystyle 1+\frac{1}{2}}}}}.
\]
The above is an expansion of length five: $5/6=[4/4,3/3,2/2,1/1,1/2]$. However, if we track instead the $p_n$ and $q_n$ without every reducing by shared factors, we would find $p_5/q_5=120/144$. For $\frac{n-1}{n}$ we can find such longest-length expansions by using the equation $\frac{n-1}{n}=\frac{n-2}{n-2+ \frac{n-2}{n-1}}$ recursively.

Note that for continued fractions that are \textit{not} proper (e.g. permitting negative numerators or without the requirement that $b_i \geq a_i$) the continued fractions of rational numbers can have any length. Examples are given in \cite{AW11} and \cite{KL17} where $N$-expansions are studied and \cite{KL17} where alternating $N$-expansions are found (the numerators change in a cyclical way).

Let us now retrieve some results that can be found in analogy to the regular continued fractions. For basic properties and results on regular continued fractions, see for example \cite{DK02}. We view matrices as M\"obius transformations so that
\[
A(x)=\left(\begin{matrix}
a & b \\
c & d
\end{matrix}\right)(x)=\frac{ax+b}{cx+d}.
\]
We define
\[
B_{a,b}=\left(\begin{matrix}
0 & a \\
1 & b
\end{matrix}\right)
\]
Let $x=[0;a_1/b_1,a_2/b_2,\ldots ]$, $x_n = [0;a_{n+1}/b_{n+1},\ldots]$ and let 
\begin{equation}\label{eq:Mnprod}
    M_n=M_{x,n}=B_{a_1,b_1}B_{a_2,b_2}\cdots B_{a_n,b_n}
\end{equation}
so that 
\[
c_n:=M_{n}(0)=\frac{a_1}{\displaystyle b_1+\frac{a_2}{\displaystyle b_2+\frac{a_3}{\displaystyle b_3+\ddots \frac{a_n}{b_n}}}}.
\]
Just as for the regular continued fractions we can write 
\begin{equation}\label{eq:Mpnqn}
M_{n} = \left(\begin{matrix}
p_{n-1} & p_n  \\
q_{n-1} & q_n
\end{matrix}\right)
\end{equation}

\noindent From \eqref{eq:Mnprod} and \eqref{eq:Mpnqn} we get
\begin{equation}\label{eq:Mnprodres}
    \det(M_n)=\det(B_{a_1,b_1}\cdots B_{a_n,b_n})=p_{n-1}q_{n}-p_nq_{n-1}=(-1)^n\prod_{i=1}^n a_i.
\end{equation}
We also have that $x=M_n(x_n)$ which gives us
\begin{equation}\label{eq:xasfraction}
    x=\frac{p_{n-1}x_n+p_{n}}{q_{n-1}x_n+q_{n}}.
\end{equation}
Using \eqref{eq:Mnprodres} and \eqref{eq:xasfraction} we get the following for the distance between $x$ and convergent $c_n$:
\[
x-\frac{p_n}{q_n}=\frac{p_{n-1}x_n+p_{n}}{q_{n-1}x_n+q_{n}}-\frac{p_n}{q_n}=\frac{q_n(p_{n-1}x_n+p_{n})-p_n(q_{n-1}x_n+q_{n})}{q_n(q_{n-1}x_n+q_{n})}=\frac{x_n(-1)^n\prod_{i=1}^n a_i}{q_n(q_{n-1}x_n+q_{n})}
\]
so that
\begin{equation}\label{eq:starting denominator estimate}|x-c_n|<\frac{\prod_{i=1}^{n+1}  a_i}{q_n q_{n+1}}\end{equation}
as well as the observation (similar to the realm of regular continued fractions) that $c_n<x$ for $n$ even and $c_n>x$ for $n$ odd.

As a final adaptation of standard results in regular continued fractions to proper continued fractions, we recall that there is a strong relationship between the regular continued fraction expansions of $x,1-x$ and their respective $q_n$:
\[x=[0;1/b_1,1/b_2,\ldots] \qquad \longrightarrow \qquad (1-x) = \begin{cases}[0;1/(b_2+1),1/b_3,1/b_4,\ldots] & (b_1=1)\\ [0;1/1,1/(b_1-1),1/b_2,\ldots] & (b_1 >1) \end{cases}\]
and for all $n \geq 1$ we have
\begin{align*}
q_n(1-x) &= q_{n+1}(x) &(b_1=1)\\
q_{n+1}(1-x) &=q_n(x) &(b_1>1)
\end{align*}
We have a similar result for proper continued fractions:
\begin{lemma}\label{lem: formula for PCF of 1-x}
Let $x=[0;a_1/b_1,a_2/b_2,\ldots]$. Then:
\[1-x = \begin{cases}
        [0;a_2/(b_1b_2+a_2),b_1a_3/b_3,a_4/b_4,\ldots] & (b_1=a_1)\\ 
        [0;1/1,a_1/(b_1-a_1),a_2/b_2,\ldots] & (b_1\geq 2a_1)
    \end{cases}
\]
Furthermore, for all $n \geq 1$ we have
\begin{align*}
q_n(1-x) &= q_{n+1}(x) &(b_1=a_1)\\
q_{n+1}(1-x) &=q_n(x) &(b_1\geq 2 a_1)
\end{align*}
\begin{proof}
    First suppose that $b_1\geq 2 a_1$:
    \[x = \frac{a_1}{b_1 + x_1}\]
    where we recall that $x_n = [a_{n+1}/b_{n+1},\ldots]$. We then compute
    \[1-x = \frac{(b_1-a_1) + x_1}{b_1 + x_1}=\frac{1}{1+\cfrac{a_1}{(b_1-a_1)+x_1}}=\cfrac{1}{1+\cfrac{a_1}{b_1-a_1+\cfrac{a_2}{b_2+\ddots}}}\]
    which is a valid PCF expansion as $b_1-a_1 \geq a_1$ by assumption.

    Second suppose that $a_1=b_1$, so that
    \[ x = \cfrac{a_1}{b_1+\cfrac{a_2}{b_2+x_2}}\]
    Then we compute (using $b_1-a_1=0$ in the numerator):
    \[1-x = \frac{\frac{a_2}{b_2+x_2}}{b_1+\frac{a_2}{b_2+x_2}}= \frac{a_2}{b_1 b_2 + a_2 + b_1 x_2}=\cfrac{a_2}{b_1 b_2+a_2 + \cfrac{b_1 a_3}{b_3 + \ddots}}\]

   The claims regarding the relationship between $q_n$, $q_{n+1}$ may now be proved inductively.
\end{proof}
\end{lemma}
\begin{remark}
    The situation when $a_1<b_1<2a_1$ in \Cref{lem: formula for PCF of 1-x} is remarkably unclear. However, the condition that $b_1 \geq 2a_1$ is seen to be equivalent to $x\leq 1/2$, and $a_1=b_1$ is equivalent to $x>1-1/(a_1+1)$. If $1/2<x<a_1/(a_1+1)$, then we may set $y=1-x$, and since $y<1/2$ the same results would apply\ldots to $y$, whose PCF expansion is not clear from that of $x$. The fact that the PCF expansion of $y$ has a tail that is not identical to some tail of the PCF expansion of $x$ in this circumstance is the primary obstacle to developing a theory of ``semi-proper continued fractions," an analogue of semi-regular continued fractions (in which numerators may be $\pm 1$).
\end{remark}
\section{Approximation and Classification Results}\label{sec: approximation results}
In this section we tackle two general questions:
\begin{itemize}
    \item If $p/q=p_n/q_n$ for some PCF expansion of $x$, how good of an approximation to $x$ is $p/q$? In other words, what upper bound can we place on $|x-p/q|$?
    \item If $p/q$ is a rational number which satisfies the bounds given by the above problem, does that guarantee that $p/q$ is actually a convergent to $x$ for some PCF expansion?
\end{itemize}

In both questions we will draw parallels to the study of regular continued fractions. For instance, for regular continued fractions we have the classical result that for any convergent $p_n/q_n$ we have
\[ \left| x - \frac{p_n}{q_n}\right| < \frac{1}{q_n q_{n+1}}.\]
Furthermore, for any $x$ there are infinitely many $n$ such that
\[\left| x - \frac{p_n}{q_n}\right| < \frac{1}{2q_n^2},\]
and conversely if $p/q$ is any rational number satisfying $|x-p/q|<1/(2q^2)$ then $p/q$ must be a convergent. In a sense, this dictates that convergents of regular continued fractions are not just excellent approximations to $x$, but the \textit{best possible} approximations. 

The first question is partially answered in \eqref{eq:starting denominator estimate}, but we derive an interesting and useful consequence:
\begin{theorem}\label{thm:proper convergent approximation}
If $p_n/q_n$ is a convergent of $x$ for some proper continued fraction expansion of $x$ and $n \geq 1$, then
\[ \left| x - \frac{p_n}{q_n}\right| < \frac{x}{q_n},\]
i.e.
\[ \left| q_n x - p_n \right| < x\]
\begin{proof}
    It is immediate to prove two things for all $n\geq 1$:
    \begin{itemize}
        \item $a_1 a_2 \cdots a_{n+1} < p_{n+1}$,\\
        \item $a_{n+1}/q_{n+1}<1/q_n$
    \end{itemize}
    So from \eqref{eq:starting denominator estimate}, if $n$ is odd, then $p_{n+1}/q_{n+1}<x$ and:
    \[ \left| x - \frac{p_n}{q_n}\right|<\frac{a_1 \cdots a_{n+1}}{q_n q_{n+1}}<\frac{1}{q_n} \frac{p_{n+1}}{q_{n+1}} < \frac{x}{q_n}.\]
    But if $n$ is even, $p_n/q_n<x$ and:
    \[ \left| x - \frac{p_n}{q_n}\right|<\frac{a_1 \cdots a_{n+1}}{q_n q_{n+1}}<\frac{a_1\cdots a_n}{q_n} \frac{a_{n+1}}{q_{n+1}} < \frac{p_n}{q_n}\frac{1}{q_n}<\frac{x}{q_n}.\qedhere\]
\end{proof}
\end{theorem}
The inequalities stated in the proof of \Cref{thm:proper convergent approximation} are not sharp, but are also optimal:
\begin{lemma}
    For any $\varepsilon>0$, for any positive integer $n$, there is an $x$ and choice of $a_1,\ldots,a_{n+1}$ such that
\[(1+\varepsilon)a_1 a_2 \cdots a_{n+1} > p_{n+1}, \qquad \frac{a_{n+1}}{q_{n+1}}>(1-\varepsilon) \frac{1}{q_n}.\]
\begin{proof}
    For the second claim, for $a_i,b_i$ arbitrary through index $n-1$, we have
    \[q_n = b_n q_{n-1}+a_n q_{n-1} \geq a_n q_{n-1},\]
    so if we select $a_n > (1/\varepsilon-1)$ we have
    \[\frac{q_{n-1}}{q_n}<\frac{\varepsilon}{1-\varepsilon}.\]
    Then, regardless of the choice of $a_{n+1}$, we may find an $x$ so that $b_{n+1}=a_{n+1}$. We then find:
    \begin{align*}
        \frac{a_{n+1}}{q_{n+1}}&= \frac{a_{n+1}}{b_{n+1}q_n+a_{n+1}q_{n-1}}\\
        &=\left(\frac{1}{\frac{b_{n+1}}{a_{n+1}}+\frac{q_{n-1}}{q_n}} \right)\frac{1}{q_n}\\
        &>\frac{1}{1+\frac{\varepsilon}{1-\varepsilon}}\frac{1}{q_n}\\
        &=\left(1-\varepsilon\right) \frac{1}{q_n}
    \end{align*}
    For the first claim, we will similarly assume that $x$ is chosen so that \textit{every} $b_i=a_i$. We then recursively choose the $a_i$ to be growing sufficiently fast so that (with $n$ and $\varepsilon$ fixed):
    \[1+\frac{p_{i-1}}{p_i} < (1+\varepsilon)^{1/n}.\]
    We can then compute
    \begin{align*}
        p_{n+1} &=b_{n+1}p_n+a_{n+1}p_{n-1}\\
        &=a_{n+1} \left(p_n+p_{n-1}\right)\\
        &=a_{n+1} \left(1+\frac{p_{n-1}}{p_n}\right)p_n\\
        &<\left(1+\varepsilon\right)^{1/n}a_{n+1}p_n\\
        &<\left(1+\varepsilon\right)^{2/n}a_{n+1}a_np_{n-1}\\
        &\vdots\\
        &<(1+\varepsilon)a_{n+1}a_n\cdots a_2 p_1
    \end{align*}
    And as $p_1=a_1$ the proof is complete.
\end{proof}
\end{lemma}

So we turn to the more intricate problem of classifying which $p,q$ might appear as some $p_n,q_n$.  We call some $p,q$ which satisfy $|x-p/q|<x/q$ a \textit{candidate} for being a convergent; we are taking care to \textit{not reduce shared factors} in $p/q$. From \Cref{thm:proper convergent approximation} we conclude that if $p,q$ is a candidate to be a convergent, then if $p/q<x$ it can only be a convergent of even index, while if $p/q>x$ then it can only be a convergent of odd index. We refer to these as \textit{even candidates} and \textit{odd candidates}, respectively.

\begin{lemma}
    \label{lem: unique candidate pairs}
    Let $x$ be irrational. Then for each positive integer $p$ there is both a unique $q$ so that $p,q$ is an odd candidate, and a unique $q$ for the pair to be an even  candidate. These two $q$'s are given by
    \[ \floor{\frac{p}{x}}, \floor{\frac{p}{x}}+1\]
    respectively.

    Similarly, for each $q$ there is at most one $p$ so that $p,q$ is an odd candidate, and at most one $p$ so that $p,q$ is an even candidate. These two $p$'s are given by $\floor{qx}-1$, $\floor{qx}$, respectively.
    \begin{proof}
        Let $x$ and $p$ be fixed. We will focus first on odd candidates. Since the interval $[p-x,p]$ is of length $x$, there is a unique $q$ so that
        \[ p-x < qx < p.\]
        On the one hand, division by $x$ gives
        \[\frac{p}{x}-1<q< \frac{p}{x}\]
        from which it follows that $q = \floor{p/x}$. On the other hand, division by $q$ establishes that
        \[ |x - p/q|<x/q,\]
        so $p,q$ is a candidate, and that $p/q>x$, so it can only be of odd index.

        Similar consideration to the interval $(p,p+x]$ gives the same result for even candidates of even index, and as this interval is exactly the previous interval translated by $x$ we find that $q$ has increased by one.

        The final statement follow immediately from similar considerations. When $p-1<qx<p$ we have $p=\floor{qx}+1$, but $p,q$ can only be an odd candidate if $|x-p/q|<1/q$, i.e. $p-x<qx<p$, and similarly for $p<qx<p+1$ we can only be an even candidate when $qx<\floor{qx}+x$, so these choices for $p$ are not \emph{guaranteed} to be candidates, but they are the only possible candidates.
    \end{proof}
\end{lemma}
We present an interesting corollary which would otherwise be tedious to prove directly:
\begin{corollary}
    A positive integer $q$ solves $\{qx\}<x$ if and only if
    \[ \floor{\frac{\floor{qx}}{x}}+1=q,\]
    and solves $\{q x\}>1-x$ if and only if
    \[\floor{\frac{\floor{qx}-1}{x}}=q.\]
    \begin{proof}
        Those $q$ for which $\{qx\}<x$ are exactly $q$ so that $p,q$ is an even candidate pair for some $p$, if and only if both $p=\floor{qx}$ and $q=\floor{p/x}+1$, giving the desired result. The derivation is similar for those $\{qx\}>1-x$.
    \end{proof}
\end{corollary}
\begin{lemma}
    For any even candidate pair $p,q$, $qx$ is the $p$-th return of the origin to the interval $[0,x)$ under rotation by $x$ modulo one. The $p$-th return map may then be considered as rotation by $T_p(x)$ after rescaling by $x$ and reversing orientation. Similarly for an odd candidate pair the $p$-th return to the interval $[1-x,1)$ may be viewed as (rescaled, reversed orientation) rotation by $T_p(x)$.
    \begin{proof}
    The fact that $qx$ is the $p$-th return to the relevant interval is direct from \Cref{lem: unique candidate pairs}. Note that for an even candidate pair we have $p<qx<p+x$, from which one derives $(q-1)<p/x<q$, so $\floor{p/x}=(q-1)$. Then the reversed-orientation rotation can be computed to be of length
    \[p+x-qx=p-(q-1)x=x\left(\frac{p}{x}-(q-1)\right)=xT_p(x),\]
    with the work being identical for odd candidate pairs except that $p-x<qx<p$ and therefore $\floor{p/q}=q$ in this case.
    \end{proof}
\end{lemma}
See \Cref{fig:scale of Tp and candidates} for a visual presentation of these candidate pairs with the scale of the induced rotation marked.
\begin{figure}[bht]
\center{\begin{tikzpicture}[scale=6]
\draw[|-|] (0,0)--(0.4,0);
\draw[|-|] (0.4,0)--(1,0);
\draw[|-|] (1.5,0)--(2.1,0);
\draw[|-|] (2.1,0)--(2.5,0);
\node[below]at(0,0){$p$};
\node[below]at(1,0){$p+x$};
\node[below]at(1.5,0){$p-x$};
\node[below]at(2.5,0){$p$};
\node[below]at(0.4,0){$qx$};
\node[below]at(2.1,0){$qx$};
\draw[->] (1,.05) to[bend right=15] node[above]{$xT_p(x)$}(0.4,.05);
\draw[->] (2.5,.05) to[bend right=15] node[above]{$xT_p(x)$}(2.1,.05);
\end{tikzpicture}
}
\caption{\label{fig:scale of Tp and candidates}Those $p,q$ which represent even and odd candidate pairs, respectively, with $T_p(x)$ the (scaled, reversed orientation) rotation of the $q$-th return to the respective interval modulo one.}
\end{figure}
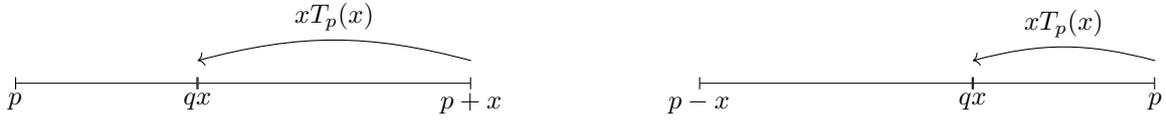
Before proceeding to a classification of which candidate $q = \floor{p/x}, \floor{p/x}+1$, or $p=\floor{qx}-1, \floor{qx}$ are actual convergents for some PCF expansion it will be useful to recall some existing terminology:

The \textit{Beatty sequence} of an irrational $r$ is the set of all integer parts $\floor{nr}$ for positive integers $n$. It follows that if $\floor{p/x}=q$, then if we set $a_1=p$, we have $b_1=q_1=q$: the Beatty sequence of $1/x$ is exactly the set of all possible $q_1$ of $x$. We may now quickly show that \emph{every} odd candidate pair is actually a convergent of odd index, specifically of index one:

\begin{proposition}\label{prop:beatty is return times}Every $p,q$ which is an odd candidate pair is some $p_1,q_1$ for a PCF expansion of $x$.
\end{proposition}
\begin{proof}
    From \Cref{lem: unique candidate pairs} for any odd candidate pair we must have $q=\floor{\frac{p}{x}}$, i.e. $q$ is the $p$-th term in the Beatty sequence of $1/x$.
\end{proof}
We therefore have a dynamic interpretation of the Beatty sequence of $1/x$: return times of the origin to the interval $[1-x,1)$ under rotation by $x$ modulo one. We similarly derive a dynamic interpretation of \textit{Rayleigh's Theorem}, which classifies exactly when two Beatty sequences form a partition of the positive integers:
\begin{theorem}[Rayleigh's Theorem]\label{thm:Rayleigh's Theorem}
    The Beatty sequences of $1/r,1/s$ partition the positive integers if and only if $s=1-r$.
    \begin{proof}
        The Beatty sequence of $1/r$ are those $n$ so that $\{nr\} \in (1-r,1)$. By reversing the direction of our rotation, we see that $\{nr\} \in [1-r,1)$ if and only if $\{n(1-r)\} \in (0,r]$ for nonzero integers $n$. Elementary considerations regarding the density of irrational rotations immediately show this set to be the exact complement of the Beatty sequence of $1/s$ if and only if $s = 1-r$.
    \end{proof}
\end{theorem}
It is worth mentioning that such pairs of Beatty sequences that partition the integers play a role in invariant and dual subtraction games, a class of games that includes Wythoff's game. In particular, every complementary pair of homogeneous Beatty sequences represents the solution to an invariant impartial game (see \cite{LHFA} and the references therein). 

\vspace{1em}

From Theorem \ref{thm:Rayleigh's Theorem} We obtain an immediate corollary:
\begin{corollary}\label{cor:Rayleigh reference}
    Every positive integer is either a possible $q_1$ of $x$ or a possible $q_1$ of $(1-x)$, but not both.
\end{corollary}

So the Beatty sequence of $1/x$ exactly classifies all possible $q_1$, and furthermore \emph{any} odd convergent of a PCF expansion of $x$ must appear as some $p_1/q_1$, but what about other $q_n$? Does something similar hold for even indices? There is a strong relationship between possible $q_n$ and possible $q_{n+2}$:
\begin{theorem}
    \label{thm:convergents subset}
    If $q=q_{k+2}$ for some PCF expansion of $x$, then $q=q_k$ for a different PCF expansion of $x$.
    \begin{proof}
        Suppose that $x=[\ldots,a_i/b_i,\ldots]$ with remainders $x_i=[a_{i+1}/b_{i+1},\ldots]$. Let $q=q_{k+2}$ for this PCF expansion of $x$. We have each of the following:
        \[\frac{a_k}{x_k} = b_k +x_{k+1}, \qquad b_{k+1} = \frac{a_{k+1}}{x_{k+1}}-x_{k+2}, \qquad a_{k+2}-b_{k+2}x_{k+2}= x_{k+2}x_{k+3}\]
        
        By applying \eqref{eq:rec for pn and qn} several times we find:
        \[q=q_{k+2}=\left(b_k \left(b_{k+2}b_{k+1}+a_{k+2} \right)+b_{k+2}a_{k+1} \right)q_{k-1} + a_k\left(b_{k+2}b_{k+1}+a_{k+2} \right)q_{k-2}\]
        Now create an alternate PCF expansion of $x=[\ldots,a_i'/b_i',\ldots]$ by setting $a_i'=a_i$, $b_i'=b_i$ for all $i \leq k-1$. If we set $a_k'=a_k(b_{k+2}b_{k+1}+a_{k+2})$, we compute $b_{k}'$ as (using our initial observations above):
        \begin{align*}
            b_k' &= \floor{\frac{a_k'}{x_k}}\\
            &=\floor{\frac{a_k}{x_k}\left(b_{k+2}b_{k+1}+a_{k+2}\right)}\\
            &=b_k \left(b_{k+2}b_{k+1}+a_{k+2}\right) + \floor{x_{k+1}\left(b_{k+2}b_{k+1}+a_{k+2}\right)}\\
            &=b_k \left(b_{k+2}b_{k+1}+a_{k+2}\right) + \floor{x_{k+1} \left(b_{k+2} \left(\frac{a_{k+1}}{x_{k+1}}-x_{k+2}\right)+a_{k+2} \right)}\\
            &=b_k \left(b_{k+2}b_{k+1}+a_{k+2}\right)+b_{k+2}a_{k+1} + \floor{x_{k+1}\left(a_{k+2}-b_{k+2}x_{k+2} \right)}\\
            &=b_k \left(b_{k+2}b_{k+1}+a_{k+2}\right)+b_{k+2}a_{k+1} + \floor{x_{k+1}x_{k+2}x_{k+3}}\\
            &=b_k \left(b_{k+2}b_{k+1}+a_{k+2}\right)+b_{k+2}a_{k+1}
        \end{align*}
        where the last line follows as each $x_i<1$, so $\floor{x_{k+1}x_{k+2}x_{k+3}}=0$. With this value of $b_k'$ we find that for our modified PCF expansion of $x$:
        \[q_{k} = b_k' q_{k-1}+a'_kq_{k-2}=q \qedhere\]
    \end{proof}
\end{theorem}
\begin{remark}\label{rem: after convegents subset theorem}
    In the above we necessarily have: $b_k \geq a_k$, $b_{k+1} \geq a_{k+1}$, $b_{k+2}\geq a_{k+2}$. Furthermore, if we set
    \begin{align*}
        x&=[\ldots,a_i/b_i,\ldots,a_k/b_k,a_{k+1}/b_{k+1},a_{k+2}/b_{k+2},x_{k+2}]\\
        &=[\ldots,a_i/b_i,\ldots,(a_k(b_{k+1}b_{k+2}+a_{k+2}))/(b_k(b_{k+1}b_{k+2}+a_{k+2})+a_{k+1}b_{k+2}),x'_k]
    \end{align*}
    then by setting the tails equal to each other:
    \[\frac{a_k(b_{k+1}b_{k+2}+a_{k+2})}{b_k(b_{k+1}b_{k+2}+a_{k+2})+a_{k+1}b_{k+2}+x'_k} = \cfrac{a_k}{b_k+\cfrac{a_{k+1}}{b_{k+1}+\cfrac{a_{k+2}}{b_{k+2}+x_{k+2}}}}\]
    By cancelling the shared factor of $a_k$ and solving for $x_{k+2}$:
    \[x_{k+2}= \frac{x'_k(b_{k+1}b_{k+2}+a_{k+2})}{a_{k+1}a_{k+2}-b_{k+1}x'_k}.\]
    Since we have $x_{k+2}<1$, we derive additionally:
    \[x'_k <\frac{a_{k+1}a_{k+2}}{b_{k+1}(b_{k+2}+1)+a_{k+2}}.\]
    It now follows that $x'_k<\frac{a_{k+1}a_{k+2}}{b_{k+1}}$, which ensures the (already known) $x_{k+2}>0$.
\end{remark}

\begin{corollary}\label{cor:which q appear later}
    If $q=q_{k}$, with $k \geq 1$, for some expansion $[\ldots,a_i/b_i,\ldots,a'_k/b'_k,x'_k]$, then $q=q_{k+2}$ for some expansion $[\ldots,a_i/b_i,\ldots,a_k/b_k,a_{k+1}/b_{k+1},a_{k+2}/b_{k+2},x_{k+2}]$ if and only if we may find positive integers $a_k,a_{k+1},a_{k+2},b_k,b_{k+1},b_{k+2}$ which solve:
    \begin{itemize}
        \item $b_k \geq a_k$, $b_{k+1} \geq a_{k+1}$, $b_{k+2} \geq a_{k+2}$,
        \item $a'_k = a_k(b_{k+1}b_{k+2}+a_{k+2})$,
        \item $b'_k = b_k(b_{k+1}b_{k+2}+a_{k+2})+a_{k+1}b_{k+2}$,
        \item $x'_k < (a_{k+1}a_{k+2})/(b_{k+1}(b_{k+2}+1)+a_{k+2})$
    \end{itemize}
    \begin{proof}
        In light of \Cref{thm:convergents subset} and \Cref{rem: after convegents subset theorem}, if such an expansion exists for which $q=q_{k+2}$, then all listed conditions are met. If we assume that such terms exist and all conditions are met, then we know
        \[x = [\ldots,a_i/b_i,\ldots,a_k/b_k,a_{k+1}/b_{k+1},a_{k+2}/b_{k+2},x_{k+2}].\]
        The first condition ensures that this is a valid PCF expansion of $x$ insofar as the $a_i/b_i$ are concerned; we still need to establish that $0<x_{k+2}<1$. If we then construct the tail
        \[\cfrac{a_k}{b_k+\cfrac{a_{k+1}}{b_{k+1}+\cfrac{a_{k+2}}{b_{k+2}+x_{k+2}}}},\]
        replicating the work of \Cref{rem: after convegents subset theorem} we set this tail equal to some
        \[ \frac{a'_k}{b'_k + x'_k},\]
        and the final condition will allow us to conclude that $0<x_{k+2}<1$, making our longer PCF expansion for $x$ valid.
    \end{proof}
\end{corollary}
\Cref{thm:convergents subset} ensures that any $q$ which appears as some denominator may appear as a denominator with lesser index but the same parity. \Cref{cor:which q appear later} places restrictions on which $q$ may appear as a $q_k$ with \textit{larger} index than already found (and same parity), providing a way to determine which subset of possible $q_k$ may appear as $q_{k+2}$. We have already remarked that the set of possible $q_{2k+1}$ is exactly the Beatty sequence of $1/x$, which is also the set of all possible $q_1$. We now also know that any $q_{2n+2}$ must appear as a possible $q_2$. 

So what can we say about the set of possible $q_2$ compared to the possible $q_1$? We observed in \Cref{lem: unique candidate pairs} that the even candidates are exactly pairs of the form $p, \floor{p/x}+1$. So suppose we set this value of $q$; we have a $p,q$ which is an even candidate. Does that automatically imply that there is some $PCF$ expansion which \emph{actually achieves} this pair as $p_2,q_2$?

Suppose first that $x$ has PCF expansion beginning $x=[0;a_1/b_1,a_2/b_2,\ldots]$. Then
\begin{equation}\label{eqn: stupid p2 and q2 reference}
    p_2 = a_1 b_2, \qquad q_2 = b_1 b_2 +a_2.
\end{equation}

The classification of which even candidate pairs $p,q$ may actually appear as some $p_2,q_2$ is surprisingly different from the case of odd index (where all candidates do appear) and also elegant in the restriction imposed:

\begin{theorem}\label{thm: nice even candidate theorem}
    For $p,q$ an even candidate pair, there is some PCF expansion of $x$ for which $p=p_2$ and $q=q_2$ if and only there is some $a$ a factor of $p$ such that
    \[T_p(x) + T_a(x) > 1.\]
    \begin{proof}
    If $p,q$ is an even candidate pair, then if there is some PCF expansion for which $p_2=p$, by \Cref{lem: unique candidate pairs} we must have $q_2=q$: it suffices just to find a PCF expansion for which $p=p_2$.
    
    First we will show that for such an even candidate pair there is such a PCF expansion if and only if for $a$ a factor of $p$ we have $p/a$ is in the Beatty sequence of $T_a(x)$. Suppose first that $p=p_2$. In light of \eqref{eqn: stupid p2 and q2 reference} we must have $a_1=a$ a factor of $p$, in which case we compute
        \[b_1 = \floor{\frac{a}{x}}, \qquad x'=T_{a}(x).\]

    Then we must find some $a_2$ so that $b_2 = p/a$, and also $b_2 = \floor{a_2/x'}$. In other words, $p/a$ must be in the Beatty sequence of $T_{a}(x)$.

    Conversely, if we may find some $a$ a factor of $p$ so that $p/a$ is in the Beatty sequence of $T_a(x)$, we simply pick $a=a_1$, then pick $a_2$ so that $b_2 = p/a$, in which case we similarly get $p=p_2$. 
        
    So that the existence of such a PCF expansion is equivalent to the existence of some $a$ a factor of $p$ such that $p/a$ is in the Beatty sequence of $T_a(x)$, i.e. from \Cref{prop:beatty is return times} there is some positive integer $a_2$ so that
    \begin{equation}\label{eqn: part of even candidate classification} \frac{p}{a}T_a(x) < a_2 < \left(\frac{p}{a}+1\right)T_a(x).\end{equation}
      
    For $p,q$ to be an even candidate we must have
    \[p<qx<p+x,\]
    where we divide by $x$ and subtract $(p/a)\floor{a/x}$ to obtain:
    \[ \frac{p}{a}T_a(x) < q-\frac{p}{a}\floor{\frac{a}{x}}<\frac{p}{a}T_a(x)+1.\]
    Therefore we must have $a_2=q-(p/a)\floor{a/x}$. So the left inequality in \eqref{eqn: part of even candidate classification} is satisfied only for this choice of $a_2$, but is then satisfied automatically for this choice of $a_2$ for any factor of $p$. Our problem is therefore equivalent to finding such an $a$ so that
    \begin{align*}
        q-\frac{p}{a}\floor{\frac{a}{x}}&<\left(\frac{p}{a}+1\right)T_a(x)\\
        q&< \frac{p}{a}\left(T_a(x)+\floor{\frac{a}{x}}\right) + T_a(x)\\
        q&<\frac{p}{x}+T_a(x)\\
        \intertext{Recall that for an even candidate $p,q$ we must have $q = \floor{p/x}+1$, so:}
        1&< \frac{p}{x}-\floor{\frac{p}{x}}+T_a(x)\\
        1&<T_p(x) + T_a(x)\qedhere
    \end{align*} 
    \end{proof}
\end{theorem}
\begin{corollary}\label{cor: cutoff for q2}
    If $p,q$ are an even candidate and $\{qx\}<\max\{x/2, xT(x)\}$ (where $T=T_1$ is the Gauss map), then there is a PCF expansion of $x$ for which $p=p_2$, $q=q_2$.
    \begin{proof}
        Refer to \Cref{fig:scale of Tp and candidates} for the scale of $T_p(x)$ in relation to $x$ itself: if $\{qx\}<x/2$, i.e. $p<qx<p+x/2$ and $\floor{p/x}=(q-1)$, then:
        \begin{align*}
            p+\frac{x}{2}&>qx\\
            \frac{p}{x}-(q-1)&>\frac{1}{2}\\
            T_p(x)&>\frac{1}{2}
        \end{align*}
        and we let $a=p$ in \Cref{thm: nice even candidate theorem}.

        Similarly, if $\{qx\}<xT(x)$:
        \begin{align*}
            p+xT(x) &>qx\\
            \frac{p}{x}-(q-1)+T(x)&>1\\
            T_p(x)+T_1(x)&>1\qedhere
        \end{align*}
    \end{proof}
\end{corollary}
\begin{remark}
    We suspect that the bounds in \Cref{cor: cutoff for q2} can not be improved. In other words, if we let $M=\max\{x/2,xT(x)\}$, for any interval proper subinterval $(\alpha,\beta) \subset (M,1)$ there exist $\{qx\}$ both where $q$ can appear as some $q_2$ in a PCF expansion of $x$, but also where $q$ can not. For such an interval there will be a positive density sequence of $q$ such that $\{qx\}$ lands in that interval. The corresponding values of $p$ ``should'' contain both $p$ with few factors such that all $T_a(x)< 1-T_p(x)$, but this sequence of $q$ ``should" also correspond to certain $p$ with very many factors, for which at least one $T_a(x)>1-T_p(x)$.
\end{remark}
\section{A Gauss Map Analogue}\label{sec: new map and ergodicity}
In this section we will find all proper continued fraction expansions of a given $x\in(0,1)$ dynamically. This can be done in various ways. Here, we add a second coordinate that will dictate which numerators to take in a natural way.
The map $T_a$ is the natural analogue to the Gauss map, but we need a way to associate some irrational $y \in (0,1)$ to any possible sequence of numerators. The numerators must be positive integers, and the Gauss map is already central to this consideration, so we propose pairing $x$ with some $y$ and using the \emph{regular} continued fraction partial quotients of $y$ as the numerators of $x$.

Specifically, let us define $T:[0,1]^2\rightarrow [0,1]^2$ as
\[
T(x,y) =
\begin{cases}
    \left(\frac{\floor{\frac{1}{y}}}{x} - \floor{\frac{\floor{\frac{1}{y}}}{x}}, \frac{1}{y} - \floor{\frac{1}{y}} \right) & \text{for } x\neq 0, y\neq 0, \\
    0 & \text{for } x=0 \text{ and or } y= 0.\\
\end{cases}
\]
Note that on the second coordinate we have the Gauss map. To see that for $x$ we indeed get expansions of the form \eqref{eq:propercf} let $(x,y)\in[0,1]^2$ and define $T^n(x,y)=(x_n,y_n)$. Furthermore, we define $a_1(y)= \floor{\frac{1}{y}}$, $b_1(x,y)= \floor{\frac{\floor{\frac{1}{y}}}{x}}$ and $a_n=a_1(y_{n-1})$, $b_n=b_1(x_{n-1},y_{n-1})$ for $n\geq 2$. Setting $x_0=x, y_0=y$, we have
\[
y_n=\frac{1}{y_{n-1}}-a_n \quad \text{ and } \quad x_n=\frac{a_n}{x_{n-1}}-b_n \quad  \text{ for } n\geq 1
\]
which gives
\begin{equation}\label{eq:inverses}
    y_{n-1}=\frac{1}{a_n+y_{n}} \quad \text{ and } \quad x_{n-1}=\frac{a_n}{b_n+x_{n}} \quad  \text{ for } n\geq 1.
\end{equation}
Using \eqref{eq:inverses} iteratively we find
\begin{equation*}
    y= \frac{1}{\displaystyle a_1+\frac{1}{\displaystyle a_2+\ddots \frac{1}{\displaystyle a_n+\ddots}}} \quad \text{ and } \quad  x= \frac{a_1}{\displaystyle b_1+\frac{a_2}{\displaystyle b_2+\ddots \frac{a_n}{\displaystyle b_n+\ddots}}}
\end{equation*}
where we stop, once we hit $0$\footnote{if $y_n=0$ and $x_n\neq 0$ for some $n$, then we did not find a continued fraction for $x$. This can only happen when $y$ is rational. This ambiguity does not affect our results and will be left out of the discussion.}. Since $x\in[0,1]$ we find that $\floor{\frac{\floor{\frac{1}{y}}}{x}}\geq  \floor{\frac{1}{y}} $ which gives $b_n\geq a_n$, and certainly $0 \leq x_n <1$, so we find a proper continued fraction of $x$. For $y$ we just develop its regular continued fraction expansion. Since any sequence of positive integers is a regular continued fraction of some $y\in[0,1]$, by using the map $T$ we can find all proper continued fraction expansions of $x$, including the regular continued fraction by taking $y=\frac{\sqrt{5}-1}{2}$. For the continued fraction of $y$ we use the shorthand notation $y=[0;a_1,a_2,\ldots]$ and for the corresponding continued fraction of $x$ we write $x=[0;a_1/b_1,a_2/b_2,\ldots ]_y$.

The map $T$ has cylinders which are rectangles in $[0,1]^2$. In the $y$-coordinate we simply have the cylinders of the Gauss map: intervals of the form $[1/(n+1),1/n]$ for $n\geq 1$. Assuming $1/(n+1) <y<1/n$, the cylinders in the $x$-coordinate are intervals of the form $[n/(m+1),n/m]$ for $m \geq n$. We denote the corresponding rectangle as $\Delta\binom{n}{m}$, those $(x,y)$ so that a proper continued fraction of $x$ begins $n/(m+\cdots)$. See \Cref{fig:slowmap} for a diagram. The map $T$ maps each of these rectangles $\Delta\binom{n}{m}$ $1:1$ back to the unit square.

\begin{figure}[bht]
\center{\begin{tikzpicture}[scale=10]
\draw (0,0)node[below]{$0$}--(1,0)node[below]{$1$}--(1,1)--(0,1)--(0,0);
\draw (1/2,0)node[below]{ $\frac{1}{2}$};
\draw (1/3,0)node[below]{ $\frac{1}{3}$};
\draw (1/4,0)node[below]{ $\frac{1}{4}$};
\node[below] at (0.65,0.8){\large $\Delta\binom{1}{1}$};
\node[below] at (0.42,0.8){ $\Delta\binom{1}{2}$};
\node[below] at (0.29,0.8){\tiny $\Delta\binom{1}{3}$};
\node[below] at (0.83,0.46){\large $\Delta\binom{2}{2}$};
\foreach \n in {2,...,15}{%
	\draw(0,1/\n)--(1,1/\n);
        \foreach \k in {0,...,15}{%
      \draw({(\n-1)/(\k+\n)},1/\n)--({(\n-1)/(\k+\n)},{1/(\n-1)});   
        }
}

\end{tikzpicture}
}

\caption{\label{fig:slowmap}The cylinders of $T$.}
\end{figure}
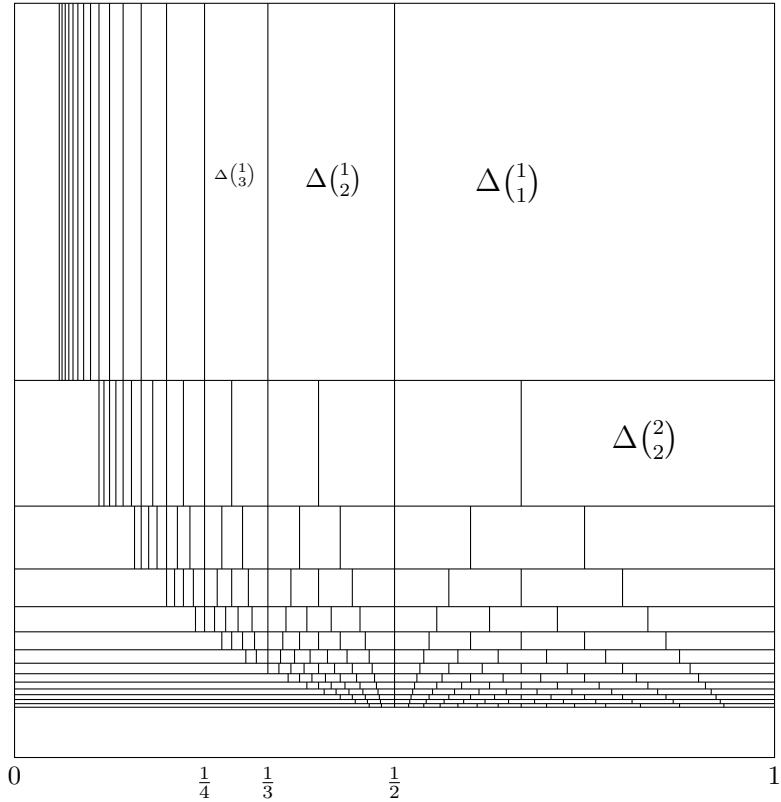

Note that for all $x\in(0,1)\backslash \mathbb{Q}$ we have uncountably many proper continued fractions.
\begin{theorem}
    The map $T(x,y)$ is ergodic with respect to a measure mutually absolutely continuous to Lebesgue measure. Moreover, it is even exponentially mixing.
    \begin{proof}
        We follow the setup of \cite{1078-0947_2005_4_639}[Theorem 1.1]. Similar to our use of this technique in \cite{langeveld2023renormalization}[Theorem 2.6], two of four conditions (`summability of upper floors' and `large image property') are trivially satisfied by the fact that $T$ maps each $\Delta\binom{n}{m}$ to the entire unit square; these rectangles form our Markov partition. The two remaining conditions are `summable variation' and the fact that these rectangles form a generating partition.

        To show that the preimages of the partition separates our space into points, first note that in the second coordinate we simply apply the Gauss map $g$ and our partition exactly separates the space in the second coordinate by the first partial quotient of $y$. So if $(x_1,y_1)$, $(x_2,y_2)$ are two points with $y_1 \neq y_2$, iterating $T$ will eventually send these two points into different elements of our partition.

        If $y_1=y_2=y$, then our map generates a proper continued fraction expansion of $x_1$, $x_2$ with the sequence of numerators given by the regular continued fraction partial quotients of $y$. By \cite{L40}[Theorems 1.1, 1.3], if $x_1 \neq x_2$ then eventually these expansions must be different, in which case our map $T$ will send the two points to different elements of the partition.

        The only nontrivial condition to verify is `summable variation':
        \[\sum_{n \in \mathbb{N}} \omega_n < \infty,\]
        where the $\omega_n$ are given by
        \[\omega_n = \sup_{C \in \mathcal{R}^n} \sup_{(x_1,y_1),(x_2,y_2) \in C} \log \frac{\det DT(x_1,y_1)}{\det DT(x_2,y_2)}\]
        where $\mathcal{R}$ is our Markov partition and  $\mathcal{R}^n=\bigvee_{i=0}^{n-1}T^{-i}(\mathcal{R})$.

        It is quick to compute
        
        \[DT(x,y) = \left[\begin{array}{c c} \frac{-\floor{\frac{1}{y}}}{x^2} & 0 \\ 0 & \frac{-1}{y^2}\end{array} \right], \qquad \det DT(x,y)=\frac{\floor{\frac{1}{y}}}{x^2y^2}.\]

        \[\frac{\det DT(x_1,y_1)}{\det DT(x_2,y_2)} = \frac{\floor{\frac{1}{y_1}} x_2^2 y_2^2}{\floor{\frac{1}{y_2}}x_1^2 y_1^2}\]
        
        Two points belong to the same element of $\mathcal{R}^n$ exactly when $y_1,y_2$ have the same first $n$ partial quotients in their regular continued fraction expansion and $x_1,x_2$ have the same first $n$ partial quotients in their proper continued fraction expansion using the partial quotients of $y$ as numerators. So $\floor{\frac{1}{y_1}} = \floor{\frac{1}{y_2}}$, allowing us to simplify
        \[ \log \frac{\det DT(x_1,y_1)}{\det DT(x_2,y_2)} = 2 \left(\log \frac{x_2}{x_1} + \log \frac{y_2}{y_1}\right) \]
        So within a particular element $C \in \mathcal{R}^n$, we find the supremum of the above by choosing $y_1,y_2$ as well as $x_1,x_2$ to be opposite ends of the cylinder of length $n$ determined by $C$. Furthermore, the magnitude of these ratios will be maximized by making ensuring these cylinders are as `wide' as possible, in the sense that there are points as far apart as possible whose continued fraction expansions (regular for $y$, proper for $x$) begin as determined by $C$.

        It follows, then (very similar to our proof in \cite{langeveld2023renormalization}) that these ratios are maximized by setting $C$ to be such that all $y \in C$ have regular continued fraction beginning $[0;1,1,\ldots,1]$. This means that $C$ will also compute regular continued fraction expansion of the $x$, and we will maximize the ratio by imposing the same condition on the first $n$ partial quotients of $x$.

        Altogether, we pick both $x_1,x_2$ and $y_1,y_2$ to be ratios of two successive Fibonacci numbers:
        \[x_1 =y_1= F_{n+2}/F_{n+1},\qquad x_2=y_2 = F_{n}/F_{n+1}\]
        when $n$ is odd, reversing the choice when $n$ is even. We then have for sufficiently large $n$ that
        \[\log \frac{x_1}{x_2}, \log \frac{y_1}{y_2} < C^n, \qquad \omega_n < 4 C^n \] where $(-1+\sqrt{5})/2<C<1$, and summability of $\omega_n$ now follows directly.

        We also remark that because the $\omega_n$ are exponentially decaying, $T$ is not just ergodic, but exponentially mixing as well \cite{1078-0947_2005_4_639}[Theorem 1.1].
    \end{proof}
\end{theorem}

\begin{corollary}
    There is some $C$ so that for almost all $(x,y)$, setting $p_n/q_n$ to be the (not necessarily coprime) convergents to $x$ given by \Cref{eq:rec for pn and qn}
    we have
    \[\lim_{n \rightarrow \infty} q_n^{1/n} = C.\]
    That is, the denominators of the proper continued fraction expansion generically grow according a fixed exponential rate.
    \begin{proof}
        Recall that 
        \[
        B_{a,b}=\left(\begin{matrix}
        0 & a \\
        1 & b
        \end{matrix}\right)
        \]
        and 
        \begin{equation*}
    M_n=B_{a_1,b_1}B_{a_2,b_2}\cdots B_{a_n,b_n}=\left(\begin{matrix}
p_{n-1} & p_n  \\
q_{n-1} & q_n
\end{matrix}\right)
\end{equation*}

        The matrix $B_{a_n,b_n}$ has eigenvalues
        \[ \lambda_{\pm}=\frac{b_n \pm \sqrt{b_n^2 + 4a_n}}{2}.\]
        We need only show $\log$-integrability of $\lambda_+$ and $1/\lambda_-$ and cite the multiplicative ergodic theorem to complete the proof. As a function of $(x,y)$, if we set $b=\floor{1/y}$ and $a = \floor{b/x}$, then we see that these eigenvalues are constant on each element of our Markov partition. For a given choice of $a,b$ the corresponding element of the Markov partition is of measure
        \[ \left(\frac{1}{a}-\frac{1}{a+1} \right) \cdot \left(\frac{a}{b}-\frac{a}{b+1} \right)=\frac{1}{b(b+1)(a+1)}.\]
        
        On this piece of the Markov partition we investigate $\log \left(\lambda_+\right)$ using the fact that $0<a\leq b$:
        \[\log b< \log\left( \frac{b+\sqrt{b^2+4a}}{2} \right)<\log(b+1)\]
        We use the elementary inequality that for any $\varepsilon>0$, for sufficiently large $a$ we will have $\log(b+1)<b^{\varepsilon}$ and therefore:
        \[ \sum_{b=a}^{\infty} \frac{\log(b+1)}{b(b+1)}< a^{-1+\varepsilon}.\]
        We may now divide this quantity by $a+1$ as determined by the measure of our rectangles and sum over $a=1,2,\ldots$ to achieve log-integrability of $\lambda_+$. The procedure for $1/\lambda_-$ is similar.
        
    \end{proof}
\end{corollary}

\section{Particular cases of proper continued fractions}\label{sec:part}
In this section we will show that for certain choices of $y$, you can find continued fractions that come from known continued fraction algorithms: `continued fraction expansions with variable numerators' (\cite{DKL15}), Engel continued fractions (\cite{HKS02}), or greedy $N$-continued fraction expansions (\cite{BGKW08}).

Other choices of $y$ may give rise to interesting (but to our knowledge unstudied) types of continued fractions. For example, the choice of $y=x$ leads to taking the digits of the regular continued fraction of $x$ as numerators. One can take $y=[0;1,2,3,4,5,\ldots]$ or any other simple sequence. Or one  can make the next numerator a function of the current digit (for Engel continued fractions they are equal). It would be interesting to see if such choices lead to any intriguing behaviour. Such investigations are beyond the scope of this article.

\subsection{Continued fraction expansions with variable numerators}\label{sec:varnum}
These continued fractions where introduced in \cite{DKL15}. Here, instead of taking $ \lfloor\frac{1}{x}\rfloor $ as the digit, it is chosen as the numerator and the digit is chosen such that the result is a proper continued fraction. More precisely, let $\hat{T}:[0,1]\rightarrow[0,1]$ be defined as 
\[
\hat{T}(x)=\frac{\lfloor\frac{1}{x}\rfloor}{x}-\left\lfloor \frac{\lfloor\frac{1}{x}\rfloor}{x}\right\rfloor 
\]
for $x\neq 0$ and $\hat{T}(0)=0$. Then, by setting $a_1(x)=\lfloor\frac{1}{x}\rfloor $, $b_1(x)=\left\lfloor \frac{\lfloor\frac{1}{x}\rfloor}{x}\right\rfloor $, $a_n(x)=a_1(\hat{T}^{n-1}(x))$ and $b_n(x)=b_1(\hat{T}^{n-1}(x))$ for $n\geq 2$, we find that $\hat{T}$ generates proper continued fractions as in \cite{DKL15}. They are exactly those continued fractions for which $a_n\in\mathbb{N}$ and $a_n\leq b_n \leq a_n^2+a_n-1$.
Our map $T(x,y)$ generates all proper continued fraction so there is a unique $y$ that gives the `continued fraction with variable numerator' from \cite{DKL15}. We can view this $y$ as a function of $x$. In Figure \ref{fig:varnum} a plot is given. The function $y(x)$ shows some interesting behaviour. First, note that almost all irrational numbers in $(0,1)$ are the image of infinitely many originals. This stems from the fact that whenever the $k^{\text{th}}$ digit in the continued fraction expansion of $y$ is greater than 1, i.e. we use a numerator $a_k>1$ in our expansion of $x$, there is more than one choice for the value of $b_k$ for $x$. Second, we have that $y(\hat{T}(x))=G(y(x))$ where $G$ is the Gauss map. This gives us the following self similarity. For $x\in(\frac{1}{2},1)$ we have $a_1=b_1=1$ so that $y\in(\frac{1}{2},1)$ and $y(\hat{T}(x))=y(\frac{1}{x}-1)=\frac{1}{y(x)}-1 =G(y(x))$. In other words: $y(x)=\frac{1}{y(\frac{1}{1+x})}-1$. More generally we have that, when $a_1(x)=j, b_1(x)=k$, the equation $y(\frac{j}{x}-k)=\frac{1}{y(x)}-j$ holds. This gives us  $y(x)=\frac{1}{y(\frac{j}{k+x})}-j$.

\begin{figure}[h]
    \centering
    \includegraphics[width=0.5\linewidth]{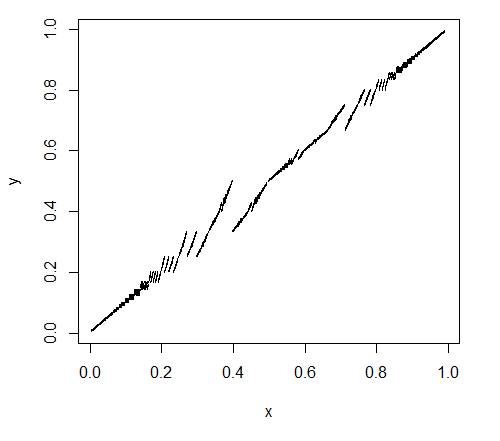}
    \caption{A plot of $y(x)$ for the continued fractions with variable numerators.}
    \label{fig:varnum}
\end{figure}

\subsection{Engel continued fractions}

Engel continued fractions where introduced in \cite{HKS02}.
They are proper continued fractions for which the numerators are a non decreasing sequence, and the digits are equal to the next numerator, i.e $b_i=a_{i+1}$ for all $i\geq 1$. These continued fractions can be generated by the map $\Tilde{T}:[0,1]\rightarrow[0,1]$ defined as
\[
\Tilde{T}(x)=\frac{1}{\lfloor\frac{1}{x}\rfloor} \left(\frac{1}{x}-\left\lfloor\frac{1}{x}\right\rfloor\right)=\frac{1}{\lfloor\frac{1}{x}\rfloor x}-1
\]
whenever $x\neq 0$ and $\Tilde{T}(0)=0$. By setting $b_1(x)=\floor{1/x}$ and $b_n=b_1(\Tilde{T}^{n-1}(x))$ we find the continued fraction
\begin{equation}\label{eq:engel}
    x= \frac{1}{\displaystyle b_1+\frac{b_1}{\displaystyle b_2+\ddots \frac{b_n}{\displaystyle b_{n+1}+\ddots}}}
\end{equation}
for which we have $b_{i+1}\geq b_{i}$.

\begin{figure}
    \centering
    \includegraphics[width=0.45\linewidth]{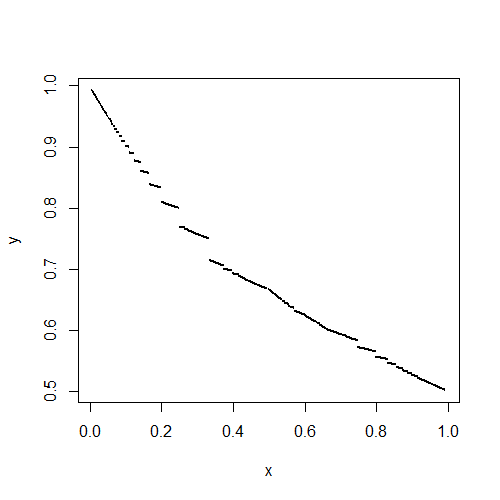}
    \caption{A plot of $y(x)$ for Engel continued fractions.}
    \label{fig:engel}
\end{figure}

Here we can also study the function $y(x)$ where $y$ is the unique value for which our map $T(x,y)$ generates the Engel continued fraction of $x$. Note that the first digit of $y$ must be 1 for all $x$ so that $y(x)>\frac{1}{2}$.
The function $y(x)$ behaves very differently from the one in Section \ref{sec:varnum}. For instance, $y(x)$ only attains values in the set of numbers with non decreasing continued fraction digits (which is a Lebesgue measure 0 set). Since the first numerator is always equal to $1$, if we want to know the $y$ value for $\Tilde{T}(x)$, we have to remove the digit 1 and the first digit from $y(x)$ and then put the digit $1$ back in front again. This gives us $y(\Tilde{T}(x))=\frac{1}{1+G(\frac{1}{y(x)}-1)}$. When $b_1(x)=k$, we find $y(\frac{1}{kx}-1)=\frac{1}{1-k+1/(\frac{1}{y(x)}-1)}$.

\subsection{Greedy $N$-continued fraction expansions}
Let $N\in \mathbb{N}_{\geq 2}$ and consider expansions of the form
\begin{equation}\label{eq:Nexp}
    x= \frac{N}{\displaystyle b_1+\frac{N}{\displaystyle b_2+\ddots \frac{N}{\displaystyle b_n+\ddots}}}
\end{equation}
which are called $N$-continued fraction expansions. It turns out that such expansions are not unique. For irrational $x$ there are infinitely many different sequences $(b_n)_{n\geq 1}$ of digits such that $x$ can be expressed as \eqref{eq:Nexp}. Though, if we demand that $b_i\geq N$ for all $i\geq n$, we find only find one representation which is called the greedy $N$-continued fraction expansion. Note that this is the only proper continued fraction with constant numerators (greater than 1). These continued fractions can easily be found using the map $T(x,y)$ and taking $y=[0;\overline{N}]$. 

\bibliographystyle{alpha}
\bibliography{bibliography}

\end{document}